\documentclass[12pt]{amsart}
\usepackage{amsfonts, amsmath, amsthm, amssymb}
\usepackage[margin=1in]{geometry}
\usepackage{hyperref}

\title{Congruences for Coefficients of Modular Functions}
\author{Paul Jenkins}
\author{DJ Thornton}
\date{\today}

\numberwithin{equation}{section}

\newtheorem{thm}{Theorem}
\newtheorem*{thmm}{Theorem}
\newtheorem{lem}{Lemma}
\newtheorem*{lemm}{Lemma}
\newtheorem{cor}{Corollary}

\begin{document}

\maketitle

\begin{abstract}
We examine canonical bases for weakly holomorphic modular forms of weight $0$ and level $p = 2, 3, 5, 7, 13$ with poles only at the cusp at $\infty$. We show that many of the Fourier coefficients for elements of these canonical bases are divisible by high powers of $p$, extending results of the first author and Andersen.  Additionally, we prove similar congruences for elements of a canonical basis for the space of modular functions of level $4$, and give congruences modulo arbitrary primes for coefficients of such modular functions in levels 1, 2, 3, 4, 5, 7, and 13.
\end{abstract}

\section{Introduction and Statement of Results}\label{sec:intro}
A holomorphic modular form of level $N$ and weight $k$ is a function $f(z)$ which is holomorphic on the complex upper half-plane,  satisfies the modular equation
\[
f\left(\frac{az+b}{cz+d}\right) = (cz+d)^kf(z) {\rm \ for\ all\ } {\small \begin{pmatrix}a & b \\ c & d\end{pmatrix}} \in \Gamma_0(N),
\]
and is holomorphic at the cusps of $\Gamma_0(N)$.  Here, as usual, \[\Gamma_0(N) = \left\{\begin{pmatrix}a\ b \\ c\ d\end{pmatrix} \in SL_2(\mathbb{Z}) : c \equiv 0\pmod{N}\right\}.\] If $f(z)$ is meromorphic at the cusps of $\Gamma_0(N)$, then we say $f$ is a weakly holomorphic modular form; additionally, if $f$ is weakly holomorphic of weight zero, we say $f$ is a level $N$ modular function. We denote by $M_k(N)$ the space of holomorphic level $N$ modular forms and by $M_k^!(N)$ the space of weakly holomorphic modular forms of level $N$. As a subspace of $M_k^!(N)$, we define the space $M_k^\sharp(N)$ to be the space of all modular forms of weight $k$ and level $N$ which are holomorphic except possibly at the cusp at $\infty$.

Every modular form has a Fourier expansion $f(z)=\sum_{n=n_0}^\infty a(n)q^n$, where $q = e^{2\pi iz}$; the coefficients $a(n)$ often encode arithmetic information and have been widely studied.  As an example, the classical $j$-invariant $j(z) = q^{-1} + \sum c(n)q^n$ is a modular function for $SL_2(\mathbb{Z})$.  In 1949, Lehner proved~\cite{Lehner1},~\cite{Lehner2} that its Fourier coefficients $c(n)$ satisfy the congruence
\[
c(2^a3^b5^c7^d) \equiv 0 \pmod{2^{3a+8}3^{2b+3}5^{c+1}7^d},
\]
showing that many of the coefficients $c(n)$ are divisible by large powers of small primes.   
Kolberg~\cite{Kolberg},~\cite{Kolberg2} and Aas~\cite{Aas} refined Lehner's work to give stronger congruences for the coefficients $c(n)$ modulo large powers of $p$ for $p \in \{2,3,5,7\}$. In~\cite{Griffin}, Griffin further extended these results by proving such congruences for every function in a canonical basis for $M_0^!(1)$.

For higher levels, Lehner showed that similar congruences hold for the coefficients of modular functions in $M_0^\sharp(p)$ with $p \in\{2, 3, 5, 7\}$ if the functions have integral Fourier coefficients and the the order of the pole at infinity is bounded appropriately.  
Andersen and the first author~\cite{Nick} extended Lehner's theorem to include all elements of a canonical basis for $M_0^\sharp(p)$, proving the following congruences, from which Lehner's results follow as a corollary. 
\begin{thmm}[\cite{Nick}, Theorem 2]
Let $p \in\{2,3,5,7\}$, and let $f_{0, m}^{(p)}(z) \in M_0^\sharp(p)$ be the unique weakly holomorphic modular form with Fourier expansion
\[
f_{0,m}^{(p)}(z) = q^{-m}+\sum\limits_{n=1}^\infty a_0^{(p)}(m,n)q^n.
\]
Suppose that $m=p^\alpha m'$ and $n = p^\beta n'$ with $(m',p)=1$ and $(n', p) = 1$. Then for $\beta > \alpha$, we have
\begin{align*} \nonumber
a_0^{(2)}(2^\alpha m',2^\beta n') \equiv 0 & \pmod{2^{3(\beta-\alpha)+8}} & {\rm if}\ p=2, \\
a_0^{(3)}(3^\alpha m',3^\beta n') \equiv 0 & \pmod{3^{2(\beta-\alpha)+3}} & {\rm if}\ p=3, \\
a_0^{(5)}(5^\alpha m',5^\beta n') \equiv 0 & \pmod{5^{(\beta-\alpha)+1}} & {\rm if}\ p=5, \\
a_0^{(7)}(7^\alpha m',7^\beta n') \equiv 0 & \pmod{7^{(\beta-\alpha)}} & {\rm if}\ p=7.
\end{align*}
\end{thmm}

Since this theorem gives congruences only for $\beta > \alpha$, it is a natural question whether similar congruences hold for the other coefficients. A quick glance at $f_{0,4}^{(2)}(z)$ shows that
\[
f_{0,4}^{(2)}(z) = q^{-4}- 196608q + 21491712q^2 - 864288768q^3 + \cdots. 
\]
For these first few coefficients, $\alpha = 2$ and $\beta < \alpha$, so the hypotheses of the theorem are not satisfied. Yet
\begin{align*}
196608 &=  2^{16}\cdot3, \\
21491712 &=  2^{12}\cdot3^2\cdot11\cdot53, \\
864288768 &=  2^{18}\cdot3\cdot7\cdot157.
\end{align*}
From this and other examples, it appears that when $\alpha > \beta$, the corresponding coefficients are also divisible by high powers of $p$.  The main result of this paper confirms this observation.
\begin{thm}\label{thm:main}
Let $p \in\{2, 3, 5, 7, 13\}$ and let $f_{0,m}^{(p)}(z) = q^{-m} + \sum\limits_{n=1}^{\infty} a_0^{(p)}(m,n)q^n $ be a weakly holomorphic modular form in $M_0^\sharp(p)$.  Let $m = p^\alpha m'$ and $n = p^\beta n'$ with $m', n'$ not divisible by $p$. Then for $\alpha > \beta$, we have
\begin{align*}
a_0^{(2)}(2^\alpha m',2^\beta n') \equiv 0 & \pmod{2^{4(\alpha-\beta)+8}} & {\rm if}\ p=2, \\
a_0^{(3)}(3^\alpha m',3^\beta n') \equiv 0 & \pmod{3^{3(\alpha-\beta)+3}} & {\rm if}\ p=3, \\
a_0^{(5)}(5^\alpha m',5^\beta n') \equiv 0 & \pmod{5^{2(\alpha-\beta)+1}} & {\rm if}\ p=5, \\
a_0^{(7)}(7^\alpha m',7^\beta n') \equiv 0 & \pmod{7^{2(\alpha-\beta)}} & {\rm if}\ p=7, \\
a_0^{(13)}(13^\alpha m',13^\beta n') \equiv 0 & \pmod{13^{\alpha-\beta}} & {\rm if}\ p = 13.
\end{align*}
\end{thm}
\noindent We remark that Theorem~\ref{thm:main} includes a congruence for $p=13$, while, as noted in~\cite{Nick}, for $\beta > \alpha$ the analogous result is a trivial congruence modulo $13^{0(\beta - \alpha)}$.  Additionally, we note that the theorem makes no divisibility predictions when $\alpha = \beta$.

When such a canonical basis is defined for $M_0^\sharp(4)$, similar congruences hold, giving the following theorem.
\begin{thm}\label{thm:2to4}
Let $f_{0,m}^{(4)}(z) = q^{-m} + \sum\limits_{n=1}^\infty a_0^{(4)}(m,n)q^{n}$ be a weakly holomorphic modular form in $M_0^\sharp (4)$. Let $m=2^\alpha m'$ and $n=2^\beta n'$ with $m',n'$ odd. Then
\[
a_0^{(4)}(2^\alpha m',2^\beta n') \equiv \begin{cases} 0 \pmod{2^{4(\alpha-\beta)+8}} \, \, {\rm if}\ \alpha > \beta, \\
 0  \pmod{2^{3(\beta-\alpha)+8}} \, \,  {\rm if}\ \beta > \alpha. \end{cases}
\]
\end{thm}
\noindent This result follows from a natural relationship between the canonical bases for $M_0^\sharp(4)$ and $M_0^\sharp(2)$.

From these theorems, it is clear that many of the coefficients of these canonical bases are divisible by high powers of primes which divide the level.  It is a natural question whether congruences exist modulo powers of primes not dividing the level. For example, consider the following modular form of level 7:
\begin{equation*}
f_{0,5}^{(7)}(z) = q^{-5}-50q-180q^2+210q^3+860q^4-1428q^5+8400q^6-3675q^7- \cdots.
\end{equation*}
It is easy to see that each of the coefficients except that of $q^5$ is divisible by 5. We prove the following theorem, which holds for any prime $p$ not dividing the level.
\begin{thm}\label{thm:arbitrary}
Let $f_{0,m}^{(N)}(z) = q^{-m} + \sum\limits_{n=1}^\infty a_0^{(N)}(m,n) \in M_0^\sharp(N)$, where $N \in\{1,2,3,4,5,7,13\}$. Let $p$ be a prime not dividing $N$, and let $r \in\mathbb{Z}^+$. If $p \nmid n$, we have
\begin{equation*}
p^r | a_0^{(N)} (mp^r,n).
\end{equation*}
\end{thm}
\noindent This result and its proof are analogous to similar divisibility results for the weights $k \in\{4,6,8,10,14\}$ appearing in~\cite{Duke}.

The rest of the paper now proceeds as follows: in Section \ref{sec:construct}, we explicitly construct canonical bases for $M_0^\sharp(p)$ and present some necessary background; Section \ref{sec:proof} proves Theorem \ref{thm:main}; Section \ref{sec:level4} describes the space $M_0^\sharp(4)$ and proves Theorem \ref{thm:2to4}; Section \ref{sec:arbitrary} contains the proof of Theorem \ref{thm:arbitrary}.

\section{Background and Canonical Bases} \label{sec:construct} 
For $p\in\{2,3,5,7, 13\}$, the congruence subgroup $\Gamma_0(p)$ has genus zero, and the space $M_0^!(p)$ is generated by powers of a single modular function known as a Hauptmodul.  A convenient Hauptmodul for $\Gamma_0(p)$ is given by
\begin{equation*}\label{eqn:Hauptmodul}
\psi^{(p)}(z) := \left(\frac{\eta(z)}{\eta(pz)}\right)^\frac{24}{p-1} =q^{-1} + O(1),
\end{equation*}
where $\eta(z)$ is the Dedekind eta function.  The modular form $\psi^{(p)}(z)$ is a modular function on $\Gamma_0(p)$ with a simple pole at $\infty$ and a simple zero at 0; its Fourier coefficients are integers. 

We define a canonical basis $\{f_{0, m}^{(p)}(z)\}_{m=0}^\infty$ for the space $M_0^\sharp(p)$ by letting $f_{0, m}^{(p)}(z)$ be the unique modular form in $M_0^\sharp(p)$ with Fourier expansion beginning $q^{-m} + O(q)$.  It is straightforward to see that $f_{0, m}^{(p)}(z)$ can be written as $F(\psi^{(p)}(z))$, where $F(x)$ is a polynomial in $x$ of degree $m$ with integer coefficients.  We write \[f_{0, m}^{(p)}(z) = q^{-m} + \sum_{n=1}^\infty a_0(m, n) q^n,\] so that $a_0(m, n)$ is the Fourier coefficient of $q^n$ in the $m$th basis element.

We note that this is an extension of the basis given in~\cite{Doud} for $M_0(p)$, and that similar bases can be defined for $M_k^\sharp(p)$ for any even weight $k$.  In all weights, these bases consist of forms $f_{k, m}^{(p)}(z)$ whose first few Fourier coefficients are $1, 0, \ldots, 0$, with the number of zeros as large as possible.  Thus, we have $f_{k, m}^{(p)}(z) = q^{-m} + \sum_{n \geq n_0} a_k^{(p)}(m, n) q^n$.

A similar construction gives a basis $\{g_{k, m}^{(p)}(z)\}$ for the subspace of $M_k^\sharp(p)$ consisting of forms which vanish at all cusps except possibly at $\infty$.  These $g_{k, m}^{(p)}(z)$ have Fourier expansion \[g_{k, m}^{(p)}(z) = q^{-m} + \sum_{n \geq n_0}b_k^{(p)}(m, n)q^n.\]  We note that in weight $0$, the only difference between these bases is the constant term, so we have 
\begin{equation} \label{abswitch} a_0^{(p)}(m, n) = b_0^{(p)}(m, n) \, \, \, \, \, \, \textrm{if} \, n \neq 0.\end{equation}

We additionally recall that for a prime $p$, the $U_p$ and $V_p$ operators (see~\cite{Atkin}) are defined as follows: for a modular form $f(z) = \sum_{n=n_0}^\infty a(n)q^n \in M_k^!(N)$, we have
\begin{align*}
U_p(f(z)) &=  \sum\limits_{n=n_0}^\infty a(pn)q^n \in M_k^{!}(pN), \\
V_p(f(z)) &=  \sum\limits_{n=n_0}^\infty a(n)q^{pn} \in M_k^{!}(pN).
\end{align*}
If $p|N$, then we actually have $U_p(f(z)) \in M_k^!(N)$, while if $p^2 | N$, then $U_p(f(z)) \in M_k^!(N/p)$.  Additionally, for a form $f(z) \in M_k^!(N)$ and a prime $p \nmid N$, the action of the standard Hecke operator $T_p$ is given by
\begin{equation}\label{Hecke}
T_p(f(z)) = U_p(f(z)) + p^{k-1}V_p(f(z)) \in M_k^{!}(N).
\end{equation}

We also make use of the Ramanujan theta operator~\cite{Ahlgren}, which acts on a modular form $f(z)$ by the rule
\begin{equation*}
\theta f(z) = q \frac{d}{dq}f(z),
\end{equation*}
so that \[\theta (\sum a(n) q^n) = \sum n a(n) q^n.\]
The theta operator maps a modular form of level $k$ to a quasi-modular form of weight $k+2$; it preserves holomorphicity but not modularity.

\section{Proof of Theorem \texorpdfstring{\ref{thm:main}}{1}}\label{sec:proof}
We begin the proof of Theorem~\ref{thm:main} with the following Zagier-type duality result for the Fourier coefficients of the basis elements $f_{k, m}^{(p)}(z)$ and $g_{k, m}^{(p)}(z)$.  This was proven in~\cite{Sharon} for levels 2 and 3; it follows from work of El-Guindy~\cite{El-Guindy}, which allows for easy extension to levels 5, 7, and 13 as well.
\begin{lemm}
Let $k$ be an even integer and let $p \in \{2, 3, 5, 7, 13\}$. For all integers $m$ and $n$, the equality
\begin{equation*}
a_k^{(p)}(m,n) = -b_{2-k}^{(p)}(n,m)
\end{equation*}
holds for the Fourier coefficients of the modular forms $f_{k,m}^{(p)}$ and $g_{2-k,n}^{(p)}$.
\end{lemm}

We will also use the following result, which is a special case of Theorem 1.1 in the paper~\cite{BOR} of Bruinier, Ono, and Rhoades.
\begin{thmm}
If $f(z) \in M_0^!(N)$, then $\theta(f) \in M_2^!(N)$.
\end{thmm}
\noindent This follows from Bol's identity and the fact that $M_0^!(N)$ is a subspace of the space of harmonic weak Maass forms.

Applying this theorem to the basis elements $f_{0, m}^{(p)}(z)$, we have the following corollary.
\begin{cor} \label{theta}
We have $\theta(f_{0, m}^{(p)}) = -m \cdot f_{2, m}^{(p)}$.
\end{cor}
\begin{proof}
Applying the $\theta$-operator to $f_{0, m}^{(p)}(z) = q^{-m} + O(q)$ gives a modular form of weight $2$ and level $p$ with Fourier expansion beginning $-mq^{-m} + O(q)$.  As the $f_{2, n}^{(p)}$ with $n \geq 0$ form a basis for $M_2^\sharp(p)$, this must be $-m f_{2, m}^{(p)}$.
\end{proof}
\noindent Looking at the action of the $\theta$-operator on the Fourier coefficients of these functions, it follows that $a_0(m,n) = \frac{-m}{n}a_2(m,n)$.

We now prove the main theorem, which we restate here for convenience.
\begin{thmm} Let $p \in\{2,3,5,7, 13\}$ and let $f_{0,m}^{(p)}(z) = q^{-m} + \sum\limits_{n=1}^\infty a_0^{(p)}(m,n)q^n \in M_0^\sharp(\Gamma_0(p))$ be an element of the basis described previously with $m=p^\alpha m'$ and $n=p^\beta n'$ with $m',n'$ not divisible by $p$. Then for $\alpha > \beta$, we have
\begin{align*}
a_0^{(2)}(2^\alpha m',2^\beta n') \equiv 0 & \pmod{2^{4(\alpha-\beta)+8}} & {\rm if}\ p=2, \\
a_0^{(3)}(3^\alpha m',3^\beta n') \equiv 0 & \pmod{3^{3(\alpha-\beta)+3}} & {\rm if}\ p=3, \\
a_0^{(5)}(5^\alpha m',5^\beta n') \equiv 0 & \pmod{5^{2(\alpha-\beta)+1}} & {\rm if}\ p=5, \\
a_0^{(7)}(7^\alpha m',7^\beta n') \equiv 0 & \pmod{7^{2(\alpha-\beta)}} & {\rm if}\ p=7, \\
a_0^{(13)}(13^\alpha m',13^\beta n') \equiv 0 & \pmod{13^{\alpha-\beta}} & {\rm if}\ p = 13.
\end{align*}
\end{thmm}
\begin{proof} Let $m=p^\alpha m'$ and $n=p^\beta n'$ with $m',n'$ not divisible by $p$. Let $a_0^{(p)}(m, n) = a_0^{(p)}(p^\alpha m',p^\beta n')$ be any coefficient where $\alpha > \beta$. Looking at the coefficient of $q^n$ in Corollary~\ref{theta}, we have $n a_0^{(p)}(m, n) = -m a_2^{(p)}(m, n)$.  On the other hand, using Zagier duality and equation~\eqref{abswitch}, we have $a_2^{(p)}(m, n) = -b_0^{(p)}(n, m) = -a_0^{(p)}(n, m)$.   Thus, we have $a_0^{(p)}(m, n)= -\frac{m}{-n}a_0^{(p)}(n, m) = \frac{p^\alpha m'}{p^\beta n'}a_0^{(p)}(n,m)$.  Note that all coefficients are integers here.

Recall that $a_0^{(p)}(n, m)$ represents the coefficient of $q^m$ in the weight zero basis element starting with $q^{-n}$. Since $\alpha > \beta$, a higher power of $p$ divides the exponent than divides the order of the pole, and we can apply Theorem 2 of~\cite{Nick}.  For instance, for $p=2$ we find that $2^{3(\alpha-\beta)+8}|a_0^{(2)}(n, m)$, and we multiply this by an extra factor of $2^{\alpha - \beta}$.  Therefore, we have
$2^{4(\alpha-\beta)+8}|a_0^{(2)}(m, n)$, as desired.  The argument for $p = 3, 5, 7, 13$ is similar.
\end{proof}

To illustrate these results, the first four basis elements for $M_0^\sharp(2)$ and $M_2^\sharp(2)$ are given below.
\begin{align*}       
f_{0,1}^{(2)}(z) & =  q^{-1} + 276q - 2048q^2 + 11202q^3 - 49152q^4 + \cdots, \\ 
f_{0,2}^{(2)}(z) & =  q^{-2} - 4096q + 98580q^2 - 1228800q^3 + 10745856q^4 - \cdots, \\ 
f_{0,3}^{(2)}(z) & =  q^{-3} + 33606q - 1843200q^2 + 43434816q^3 - 648216576q^4 + \cdots, \\ 
f_{0,4}^{(2)}(z) & =  q^{-4} - 196608q + 21491712q^2 - 864288768q^3 + 20246003988q^4 - \cdots. \\ 
        &    \\ \nonumber
f_{2,1}^{(2)}(z) & =  q^{-1} - 276q + 4096q^2 - 33606q^3 + 196608q^4 - \cdots, \\ 
f_{2,2}^{(2)}(z) & =  q^{-2} + 2048q - 98580q^2 + 1843200q^3 - 21491712q^4 + \cdots, \\ 
f_{2,3}^{(2)}(z) & =  q^{-3} - 11202q + 1228800q^2 - 43434816q^3 + 864288768q^4 - \cdots, \\ 
f_{2,4}^{(2)}(z) & =  q^{-4} + 49152q - 10745856q^2 + 648216576q^3 - 20246003988q^4 + \cdots.
\end{align*}
By comparing rows of coefficients in weight $0$ to columns of coefficients in weight $2$, the duality is clear; for example, $a_0^{(2)}(1,2) = b_0^{(2)}(1,2) = -a_2^{(2)}(2,1)$. The effect of the theta operator is also clear if the coefficients are factored; for example,
\begin{align*}
f_{0,3} & =  q^{-3} + 33606q - 1843200q^2 + 43434816q^3 - 648216576q^4 + \cdots, \\
        & =  q^{-3} + 3\cdot11202q - \frac{3}{2}\cdot1228800q^2 + 1\cdot43434816q^3 - \frac{3}{4}\cdot864288768q^4 + \cdots, \\
f_{2,3} & =  q^{-3} - 11202q + 1228800q^2 - 43434816q^3 + 864288768q^4 - \cdots.
\end{align*}

\section{Level 4}\label{sec:level4}
The group $\Gamma_0(4)$ has genus zero and 3 cusps, which can be taken to be at $0$, at $\frac{1}{2}$, and at $\infty$.  We  construct a similar canonical basis for $M_0^\sharp(4)$ by letting $f_{0, m}^{(4)}(z)$, for all $m \geq 0$, be the unique modular form in this space with Fourier expansion \[f_{0, m}^{(4)}(z) = q^{-m} + \sum_{n=1}^\infty a_0^{(4)}(m, n) q^n.\]  Similarly, we define a basis for the subspace of forms in $M_2^\sharp(4)$ which vanish both at $0$ and at $\frac{1}{2}$ by defining \[g_{2, m}^{(4)}(z) = q^{-m} + \sum_{n=0}^\infty b_2^{(4)}(m, n) q^n\] for all $m \geq 1$.  Given this notation, the first author and Haddock~\cite{Haddock} proved the following results.
\begin{thmm}[\cite{Haddock}, Theorem 2]
For all integers $m, n$, we have the duality of coefficients
\begin{equation*}
a_0^{(4)}(m,n) = -b_{2}^{(4)}(n,m).
\end{equation*}
\end{thmm}
\begin{thmm}[\cite{Haddock}, Theorem 3]
If $n \not\equiv m \pmod{2}$, then $a_0^{(4)}(m,n) = 0$.
\end{thmm}

We now describe the action of the $U_2$ operator on these basis elements.
\begin{thm}\label{thm:basis}
For any nonnegative integer $m$, we have $U_2(f_{0, 2m}^{(4)}(z)) = f_{0,m}^{(2)}(z)$ and $U_2(f_{0,2m+1}^{(4)}(z)) = 0$.
\end{thm}
\begin{proof}
Let $f_{0, 2m}^{(4)}(z) \in M_0^\sharp(\Gamma_0(4))$, and note that $V_2 U_2$ acts as the identity on $f_{0, 2m}^{(4)}(z)$ (see section 3 of~\cite{Haddock}).  Since $V_2(f_{0, m}^{(2)}(z))$ is also a modular form in $M_0^\sharp(4)$ with principal part $q^{-2m}$, the difference \[V_2\left(U_2(f_{0, 2m}^{(4)}(z)) - f_{0, m}^{(2)}(z)\right)\] is a modular form in $M_k^\sharp(4)$ which vanishes at $\infty$ and must therefore be zero.  The first result follows.

To see that $U_2(f_{k,2m+1}^{(4)}(z)) = 0$, note that the order of the pole is odd, so all of the nonzero exponents are in the Fourier expansion are odd.  Applying $U_2$, all of the terms vanish.
\end{proof}
We now prove Theorem~\ref{thm:2to4}.
\begin{thmm}
Let $f_{0,m}^{(4)}(z) = q^{-m} + \sum\limits_{n=1}^\infty a_0^{(4)}(m,n)q^{n} \in M_0^\sharp(4)$ be an element of the canonical basis.  Suppose that $m=2^\alpha m'$ and $n=2^\beta n'$ with $m'$ and $n'$ odd. Then for $\alpha \neq \beta$, we have
\begin{align*}
a_0^{(4)}(2^\alpha m',2^\beta n') \equiv 0 & \pmod{2^{4(\alpha-\beta)+8}} & {\rm if}\ \alpha > \beta, \\
a_0^{(4)}(2^\alpha m',2^\beta n') \equiv 0 & \pmod{2^{3(\beta-\alpha)+8}} & {\rm if}\ \beta > \alpha.
\end{align*}
\end{thmm}
\begin{proof}
Let $f_{0,m}^{(4)}(z) \in M_0^\sharp(4)$. By Theorem~\ref{thm:basis}, we know that $f_{0,2m}^{(4)}(z) | U_2 = f_{0,m}^{(2)}(z)$.  Looking at coefficients, we find that  
\[a_0^{(4)}(2^\alpha m',2^\beta n') = a_0^{(2)}(2^{\alpha-1} m',2^{\beta-1} n'),\] for any odd $m',n'$ and any $\alpha, \beta \geq 1$. By Theorem \ref{thm:main}, we know that for $\alpha > \beta$, we have the congruence \[a_0^{(2)}(2^{\alpha-1} m',2^{\beta-1} n') \equiv 0 \pmod{2^{4(\alpha-\beta)+8}},\] giving the first result.  
Similarly, when $\beta > \alpha$, we know by Theorem 2 in~\cite{Nick} that \[a_0^{(2)}(2^{\alpha-1} m',2^{\beta-1} n') \equiv 0 \pmod{2^{3(\beta-\alpha)+8}},\] giving the second part. 
\end{proof}
We note that this theorem applies only when $m$ is even, since when $m$ is odd, all of the exponents appearing in $f_{0, m}^{(4)}(z)$ are also odd.

\section{Arbitrary Primes}\label{sec:arbitrary}

Theorem~\ref{thm:arbitrary} will follow from the following result.
\begin{lem}
Let $N \in\{1,2,3,4,5,7,13\}$ and let $p$ be a prime not dividing $N$.  Let $f_{0,m}^{(N)}(z) = q^{-m} + \sum_{n=n_0}^\infty a_0^{(N)}(m,n) \in M_k^\sharp(N)$ be a basis element as before.  Then for any positive integer $r$ we have
\begin{equation}
p^r\left(a_0(m,np^r)-a_0\left(\frac{m}{p},np^{r-1}\right)\right) = a_0(mp^r,n)-a_0\left(mp^{r-1},\frac{n}{p}\right).
\end{equation}
\end{lem}
\begin{proof} We proceed as in Lemma 1 of~\cite{Duke}.  Applying the $T_p$ operator to the basis element $f_{0, m}^{(N)}(z)$ and using~\eqref{Hecke}, we find that the coefficient of $q^n$ in $T_p(f_{0, m}^{(N)}(z))$ is $a_0^{(N)}(m, np) + p^{-1}a_0^{(N)}(m, \frac{n}{p})$.  Additionally, applying~\eqref{Hecke} to the $q^{-m}$ term allows us to conclude that $T_p(f_{0, m}^{(N)}(z)) = p^{-1} q^{-mp} + q^{-m/p} + O(q)$, where the second term is omitted if $p \nmid m$.  A straightforward calculation similar to that in section 3 of~\cite{Haddock} shows that $T_p$ also preserves the space $M_k^\sharp(N)$, allowing us to write $T_p(f_{0, m}^{(N)}(z))$ as a sum of basis elements as \[T_p(f_{0, m}^{(N)}(z)) = p^{-1}f_{0, -mp}^{(N)}(z) + f_{0, -m/p}^{(N)}(z).\]  Thus, the coefficient of $q^n$ is also given by $p^{-1}a_0^{(N)}(mp, n) + a_0^{(N)}(\frac{m}{p}, n)$.  Combining these two expressions for the coefficient of $q^n$ in $T_p(f_{0, m}^{(N)}(z))$, we find that
\begin{equation}\label{eqn:basecase}
a_0^{(N)}(m,np)=p^{-1}\left(a_0^{(N)}(mp,n)-a_0^{(N)}\left(m,\frac{n}{p}\right)\right)+a_0^{(N)}\left(\frac{m}{p},n\right).
\end{equation}
Note that for $1 \leq i \leq r-1$, replacing $m$ with $mp^i$ and $n$ with $p^{r-i-1}n$ in~\eqref{eqn:basecase} gives
\begin{multline}\label{eqn:initialsub}
p^{-i}(a_0^{(N)}(mp^i,np^{r-i})-a_0^{(N)}(mp^{i-1},np^{r-i-1})) \\
= p^{-(i+1)}(a_0^{(N)}(mp^{i+1},np^{r-i-1})-a_0^{(N)}(mp^i,np^{r-i-2})).
\end{multline}
We now replace $n$ with $np^{r-1}$ in equation~\eqref{eqn:basecase} to obtain
\begin{equation}
a_0^{(N)}(m,np^r)=p^{-1}(a_0^{(N)}(mp,np^{r-1})-a_0^{(N)}(m,np^{r-2}))+a_0^{(N)}\left(\frac{m}{p},np^{r-1}\right),
\end{equation}
and use \eqref{eqn:initialsub} a total of $(r-1)$ times to obtain
\begin{equation}
a_0^{(N)}(m,np^r)=p^{-r}\left(a_0^{(N)}(mp^r,n)-a_0^{(N)}\left(mp^{r-1},\frac{n}{p}\right)\right)+a_0^{(N)} \left(\frac{m}{p},np^{r-1}\right).
\end{equation}
Multiplying by $p^r$ and rearranging proves the lemma.
\end{proof}
\noindent We remark that this lemma relies only on the existence of the canonical basis and the fact that the $T_p$ operator preserves the space $M_k^\sharp(N)$.

Theorem~\ref{thm:arbitrary} now follows from this lemma, noting that if $p \nmid n$, then $a_0^{(N)}(mp^{r-1}, n/p) = 0$.

\bibliographystyle{amsplain}
\providecommand{\bysame}{\leavevmode\hbox to3em{\hrulefill}\thinspace}
\providecommand{\MR}{\relax\ifhmode\unskip\space\fi MR }
\providecommand{\MRhref}[2]{%
  \href{http://www.ams.org/mathscinet-getitem?mr=#1}{#2}
}
\providecommand{\href}[2]{#2}

\end{document}